\documentclass[12pt,a4paper]{amsart}
\usepackage{amsfonts}
\usepackage{amsthm}
\usepackage{amsmath}
\usepackage{amscd}
\usepackage[latin2]{inputenc}
\usepackage{t1enc}
\usepackage[mathscr]{eucal}
\usepackage{indentfirst}
\usepackage{graphicx}
\usepackage{graphics}
\usepackage{pict2e}
\usepackage{epic}
\numberwithin{equation}{section}
\usepackage[margin=2.9cm]{geometry}
\usepackage{epstopdf} 
\usepackage{tipa}
\usepackage{accents}
\usepackage{stmaryrd}

\theoremstyle{plain}
\newtheorem{Th}{Theorem}[section]
\newtheorem{Lemma}[Th]{Lemma}
\newtheorem{Cor}[Th]{Corollary}
\newtheorem{Prop}[Th]{Proposition}

 \theoremstyle{definition}
\newtheorem{Def}[Th]{Definition}

\newtheorem{Rmk}[Th]{Remark}
\newtheorem{?}[Th]{Problem}

\begin{document}

\title[A Generalization of Thorpe's Inequality]{A Generalization of Thorpe's Inequality}

\author{Brian Klatt}

\address{Rutgers University \\ Department of Mathematics \\
New Brunswick, NJ \\ United States} 

\email{brn.kltt@gmail.com}



\begin{abstract} We present a generalization of the topological inequality of Thorpe between the Euler characteristic and $k^{th}$-Pontryagin number of a $4k$-manifold. We also correct and complete some of the arguments from the work of Thorpe in which this inequality originally appeared.
\end{abstract}

\maketitle

\section{Introduction}

The present work is a kind of commentary on the work \cite{Thorpe} of J. A. Thorpe. The lasting achievement of that research is Thorpe's discovery of a topological obstruction to a certain curvature condition for manifolds whose dimension is a multiple of four. When the dimension is equal to four, the obstruction is the \emph{Hitchin-Thorpe inequality}, $2\chi\geq |p_1|$, and the curvature condition is in that case equivalent to the requirement that the manifold be Einstein. (It is so-called because Hitchin in \cite{Hitchin} independently rediscovered the inequality in the case of dimension four and went further by classifying the topological types that can occur when equality holds.)

While trying to understand \cite{Thorpe}, I realized that I could give a proof of a somewhat more general result which seemed to me more straightforward than the argument given by Thorpe for his inequality. The main purpose of this work is the presentation of this result and its proof, which can be found in Section 3. Section 2 contains some preliminaries to aid in following the argument.

In the course of my reading, I also noticed what seemed to be an error in the discussion of the Bianchi identity. Later I found that this point is apparently discussed in the dissertation \cite{StehneyDiss} of A. Stehney, which was written under Thorpe's supervision, but I could not obtain this work. One of the main objectives of Section 4 is to discuss this issue and give a way of constructing counterexamples to the problematic claim. We also, as the second main objective of Section 4, fill a gap that is opened later in Thorpe's arguments as a result of this error; I could not identify work in the literature that filled this gap.

\section{Preliminaries} 

\subsection{Basic Objects}

The objects that we are concerned with are as follows. We have a compact oriented Riemannian manifold $(M^{4k}, g)$ of dimension $4k$, and a real oriented vector bundle $\xi\rightarrow M$ with a smooth inner product $h$ and compatible connection $D$, so that $Dh=0$. The curvature of $D$ is 
$$R(X, Y)s=(D_XD_Y-D_YD_X-D_{[X,Y]})s.$$
 If local bases $\{X_i\}$ and $\{s_a\}$ of $TM$ and $\xi$ respectively are chosen then we write $h(s_a, s_b)=h_{ab}$ and $s^a(R(X_i, X_j)s_b)=R_{ij}{}^a{}_b$ where $\{s^a\}$ is the induced dual basis of $\xi^*$. We can use $h$ to also define 
 $$R(X, Y, s, t)=h(R(X, Y)t, s)$$ 
whose components are $R_{ijab}=R(X_i, X_j, s_a, s_b)=h(R(X_i, X_j)s_b, s_a)=h(R_{ij}{}^c{}_b s_c, s_a)=h_{ac}R_{ij}{}^c{}_b$; thus our convention is to lower the upper index to the third slot. Since $D$ is compatible with $h$, the curvature satisfies 
\begin{equation}\label{skewsymR}
R(X, Y, s, t)=-R(X, Y, t, s).
\end{equation}
 With respect to the local basis $\{s_a\}$ we also have locally the curvature matrix $R^a{}_b$ of 2-forms defined by $R^a{}_b(X, Y)=s^a(R(X, Y)s_b)$. With the metric $h$ there is another local matrix of 2-forms given by $R_{ab}=h_{ac}R^c{}_b$ whose entries are generally distinct from the first, but if $\{s_a\}$ is orthonormal so that $h_{ab}=\delta_{ab}$, then $R_{ab}=R^a{}_b$. By equation (\ref{skewsymR}) we have $R_{ab}=-R_{ba}$ in any local basis but if the basis is orthonormal $R^a{}_b=-R^b{}_a$ also holds.

\subsection{Bi-Forms}

When we have a metric compatible connection $D$ as above, the curvature satisfies $R(X, Y, s, t)=-R(Y, X, s, t)=-R(X, Y, t, s)$. Thus $R\in \Lambda^2 T^*M \otimes \Lambda^2 \xi^*$ and we say that $R$ is a $(2, 2)$-bi-form. 

\begin{Def}
Suppose $V$ and $W$ are finite-dimensional vector spaces. An element of $B^{r,s}(V, W)=\Lambda^r V\otimes \Lambda^s W$ is called an $(r,s)$\emph{-bi-form}. The wedge product of bi-forms is defined by $(\phi_1\otimes \psi_1)\wedge (\phi_2\otimes \psi_2)=(\phi_1\wedge\phi_2)\otimes(\psi_1\wedge\psi_2)$ on decomposable elements of the tensor product, which extends by bilinearity to a well-defined product.

If $V=W$ we follow \cite{Kulkarni} (alternatively see \cite{Labbi}) and call elements of $D^{r, s}V=B^{r, s}(V, V)=\Lambda^r V\otimes \Lambda^s V$ the $(r,s)$\emph{-double forms}. We denote by $C^rV$ the symmetric elements of $D^{r, r}V$. They are elements fixed by the transpose operation $t$ which interchanges the two tensor factors, that is, $C^rV=S^2(\Lambda^r V)$; elements of this space are called \emph{curvature structures}.
\end{Def}

\noindent The utility for us of taking this viewpoint is that we can naturally consider the repeated products of the curvature with itself using the bi-form wedge product: $$R^p=\underbrace{R\wedge...\wedge R}_{\text{p times}} .$$
Of course, $R^p$ is a $(2p, 2p)$-bi-form. Explicitly, we have
\begin{align*}
&R^p(X_1, ..., X_{2p}, s_1, ..., s_{2p}) \\
&=\frac{1}{2^{2p}}\sum_{\sigma, \tau} (-1)^{\sigma}(-1)^{\tau} R(X_{\sigma(1)}, X_{\sigma(2)}, s_{\tau(1)}, s_{\tau(2)})...R(X_{\sigma(2p-1)}, X_{\sigma(2p)}, s_{\tau(2p-1)}, s_{\tau(2p)})
\end{align*}
since this formula just expresses that we're doing the wedge product separately and simultaneously on the $X$ and $s$ inputs.

\subsection{Star Operators}

If $V$ is a finite-dimensional vector space with an inner product $(\, ,\,)$ and $\{e_i\}$ is an orthonormal basis, we can make $\Lambda^p V$ into an inner product space by declaring the basis $\{e_{i_1}\wedge ...\wedge e_{i_p}: i_1<...<i_p\}$ to be orthonormal (and this definition is independent of the original orthonormal basis $\{e_i\}$). If in addition $V$ is oriented and the orthonormal basis $\{e_1, ..., e_n\}$ induces the orientation, we have the distinguished unit-length \emph{volume element} $\epsilon=e_1\wedge...\wedge e_n\in\Lambda^n V$ which defines the \emph{star operator} $*_p:\Lambda^p V\rightarrow \Lambda^{n-p} V$ by 
\begin{equation}\label{star}
\alpha\wedge *_p\,\beta=(\alpha, \beta)\epsilon.
\end{equation} It satisfies $*_{n-p}\,*_p=(-1)^{p(n-p)}$, so if $n=4k$, $*_{2k}$ is an involution on $\Lambda^{2k}V$. Considering all $S=\{e_I=e_{i_1}\wedge...\wedge e_{i_{2k}}: i_1<...<i_{2k}, i_1=1\}$, we have $|S|=\frac{1}{2}\binom{4k}{2k}$ and $S\cup *S$ is a basis of $\Lambda^{2k}V$ so $\Lambda^{2k}_+=\mathrm{span}\{e_I+*_{2k}e_I : e_I\in S\}$ and $\Lambda^{2k}_-=\mathrm{span}\{e_I-*_{2k}e_I : e_I\in S\}$ satisfy $\Lambda^{2k}V=\Lambda^{2k}_+\oplus\Lambda^{2k}_-$; each summand has dimension equal to $\frac{1}{2}\binom{4k}{2k}$ and clearly the elements of $\Lambda^{2k}_+$ and $\Lambda^{2k}_-$ are eigenvectors for the eigenvalues $+1$ and $-1$ respectively. These subspaces are the \emph{self-dual} and \emph{anti-self-dual} forms, respectively, and they are orthogonal subspaces, as follows from Equation~\ref{star}.

If we have two oriented inner product spaces $V$ and $W$ with star operators $*_V$ and $*_W$, they naturally extend to the operator $*=*_{V, r}\otimes *_{W, s}$ on the space $B^{r, s}(V, W)$ of $(r,s)$-bi-forms. When $V=W$ we don't even need $V$ to be oriented to define $*$ on the ring of double forms $D^{r,s}V$ since the star operator is defined up to a sign on an unoriented vector space and this ambiguity is cancelled in the tensor product.

Now suppose $V$ is an oriented inner product space of dimension $4k$ with star operator $*_V$. If $\alpha\in \Lambda^{2k}V$ we can write $\alpha=\alpha_+ + \alpha_-$ corresponding to $\Lambda^{2k}V=\Lambda^{2k}_+\oplus\Lambda^{2k}_-$. We have the identity 
\begin{equation}\label{wedgesquare}
\alpha\wedge\alpha=(|\alpha_+|^2-|\alpha_-|^2)\epsilon
\end{equation}
 as an easy consequence of (\ref{star}). Now suppose $W$ is another oriented inner product space with star operator $*_W$, and we have a bi-form $T\in \Lambda^{2k} V\otimes\Lambda^q W$ where $\dim(V)=4k$. Then $T\in (\Lambda^{2k}_+ V\otimes\Lambda^q W)\oplus (\Lambda^{2k}_- V\otimes\Lambda^q W)$ so we can write $T=T^+ + T^-$ corresponding to this decomposition. If instead $T\in \Lambda^p V\otimes\Lambda^{2l} W$ where $\dim(W)=4l$ then $T\in (\Lambda^p V\otimes\Lambda^{2l}_+ W)\oplus (\Lambda^p V\otimes\Lambda^{2l}_- W)$ so we can accordingly write $T=T_+ + T_-$. If $\dim(V)=4k$, $\dim(W)=4l$ and $T\in \Lambda^{2k} V\otimes\Lambda^{2l} W$ then $T=T^+_+ + T^+_- + T^-_+ + T^-_-$ according to the decomposition $\Lambda^{2k} V\otimes\Lambda^{2l} W=(\Lambda^{2k}_+ V\otimes\Lambda^{2l}_+ W) \oplus (\Lambda^{2k}_+ V\otimes\Lambda^{2l}_- W) \oplus (\Lambda^{2k}_- V\otimes\Lambda^{2l}_+ W) \oplus (\Lambda^{2k}_- V\otimes\Lambda^{2l}_- W)$. Each of the decompositions discussed here is orthogonal. When $T=T^+_+ + T^+_- + T^-_+ + T^-_-$, the extension of Equation~\ref{wedgesquare} given by
 \begin{equation}\label{wedgesquare2}
 T^2=(|T^+_+|^2 - |T^+_-|^2 - |T^-_+|^2 + |T^-_-|^2)\epsilon\otimes\varepsilon
 \end{equation}
 is easily verified, where $\epsilon$, $\varepsilon$ are the volume elements for $V$, $W$ respectively.

\section{Main Computations}

 Our objective is to compute the $k^{th}$-Pontryagin class $p_k(\xi)$ and the Euler class $e(\xi)$.

\subsection{Pontryagin Class}

We assume in this section that $\xi$ has rank $l$ where $l\geq 2k$. The standard expression for a representative of the Pontryagin class from Chern-Weil theory \cite{KN} is the following normalization of the $(2k)^{th}$-elementary symmetric polynomial applied to a local curvature matrix of 2-forms:

\[p_k(\xi)=\frac{1}{(2\pi)^{2k}}\displaystyle\sum_{1\le a_1<...<a_{2k}\le l}\displaystyle\sum_{\sigma\in S_{2k}}(-1)^{\sigma}\,R^{a_1}_{a_{\sigma(1)}}\wedge ... \wedge R^{a_{2k}}_{a_{\sigma(2k)}}\]

\noindent Our key observation in computing with this quantity is to note that if $\{s_a\}$ is orthonormal the inner summation is the determinant of the $2k\times 2k$ antisymmetric matrix $(R^{a_j}_{a_k})_{j, k}$. Such a determinant is, \emph{as a polynomial}, the square of the Pfaffian \[\frac{1}{2^k k!}\displaystyle\sum_{\sigma}(-1)^{\sigma}R^{a_{\sigma(1)}}_{a_{\sigma(2)}}\wedge ...\wedge R^{a_{\sigma(2k-1)}}_{a_{\sigma(2k)}}=\frac{1}{2^k k!}\displaystyle\sum_{\sigma}(-1)^{\sigma}R_{a_{\sigma(1)}a_{\sigma(2)}}\wedge ...\wedge R_{a_{\sigma(2k-1)}a_{\sigma(2k)}}\]
where we have lowered the index of the local curvature 2-form because we're now assuming our basis is orthonormal. However, we can easily recognize that
\[\frac{1}{2^k k!}\displaystyle\sum_{\sigma}(-1)^{\sigma}R_{a_{\sigma(1)}a_{\sigma(2)}}\wedge ...\wedge R_{a_{\sigma(2k-1)}a_{\sigma(2k)}}=\frac{1}{k!}R^k(s_{a_1}, ..., s_{a_{2k}})\]
in which we're evaluating $R^k$ on its last $2k$ arguments, so $R^k(s_{a_1}, ..., s_{a_{2k}})\in \Lambda^{2k}T^*M$. Continuing our computation above then and using equation (\ref{wedgesquare}),
\begin{align*}
p_k(\xi)&=\frac{1}{(2\pi)^{2k}(k!)^2}\displaystyle\sum_{1\le a_1<...<a_{2k}\le l} (R^k(s_{a_1}, ..., s_{a_{2k}}))^2\\
&=\frac{1}{(2\pi)^{2k}(k!)^2}\displaystyle\sum_{1\le a_1<...<a_{2k}\le l} (|(R^k)^+(s_{a_1}, ..., s_{a_{2k}})|^2-|(R^k)^-(s_{a_1}, ..., s_{a_{2k}})|^2)\epsilon\\
&=\frac{1}{(2\pi)^{2k}(k!)^2}(|(R^k)^+|^2-|(R^k)^-|^2)\epsilon.
\end{align*}
Thus we have the following

\begin{Prop}
If $\xi$ is a real vector bundle over $M^{4k}$ of rank $l\geq 2k$ with inner product $h$, compatible connection $D$, and curvature $R$, then
\[p_k(\xi)=\frac{1}{(2\pi)^{2k}(k!)^2}(|(R^k)^+|^2-|(R^k)^-|^2)\epsilon.\]
If the rank of $\xi$ equals $4k$, then by (\ref{wedgesquare2}),
\[p_k(\xi)=\frac{1}{(2\pi)^{2k}(k!)^2}(|(R^k)^+_+|^2 + |(R^k)^+_-|^2 - |(R^k)^-_+|^2 - |(R^k)^-_-|^2)\epsilon.\]
\end{Prop}

\subsection{Euler Class}

In this section we assume the rank of $\xi$ is $4k$. Beginning with the standard expression for the Euler class of $\xi$ in terms of the Pfaffian \cite{KN} and with respect to an oriented orthonormal basis $\{s_1, ..., s_{4k}\}$ of local sections, we have
\begin{align*}
e(\xi)&=\frac{1}{(2\pi)^{2k}2^{2k}(2k)!}\displaystyle\sum_{\sigma}\,(-1)^{\sigma}R^{\sigma(1)}_{\sigma(2)}\wedge ... \wedge R^{\sigma(4k-1)}_{\sigma(4k)}\\
&=\frac{1}{(2\pi)^{2k}2^{2k}(2k)!}\displaystyle\sum_{\sigma}\,(-1)^{\sigma}R_{\sigma(1)\sigma(2)}\wedge ... \wedge R_{\sigma(4k-1)\sigma(4k)}\\
&=\frac{1}{(2\pi)^{2k}(2k)!}R^{2k}(s_1, ..., s_{4k})\\
&=\frac{1}{(2\pi)^{2k}(2k)!}(R^k\wedge R^k)(s_1, ..., s_{4k})\\
&=\frac{1}{(2\pi)^{2k}(2k)!}(|(R^k)^+_+|^2 - |(R^k)^+_-|^2 - |(R^k)^-_+|^2 + |(R^k)^-_-|^2)\epsilon\cdot\varepsilon(s_1, ..., s_{4k})\\
&=\frac{1}{(2\pi)^{2k}(2k)!}(|(R^k)^+_+|^2 - |(R^k)^+_-|^2 - |(R^k)^-_+|^2 + |(R^k)^-_-|^2)\epsilon
\end{align*}
where we used Equation~\ref{wedgesquare2} in the second-to-last equality, and have used $\varepsilon$ to denote the volume element of $\xi$.
So we have the

\begin{Prop}
If $\xi$ is a oriented real vector bundle over $M^{4k}$ of rank $4k$ with inner product $h$, compatible connection $D$, and curvature $R$, then
\[e(\xi)=\frac{1}{(2\pi)^{2k}(2k)!}(|(R^k)^+_+|^2 - |(R^k)^+_-|^2 - |(R^k)^-_+|^2 + |(R^k)^-_-|^2)\epsilon.\]
\end{Prop}

This immediately leads to

\begin{Th}[Generalized Thorpe Inequalities]
If $\xi$ is an oriented real vector bundle over $M^{4k}$ of rank $4k$ with inner product $h$, compatible connection $D$, and curvature $R$, then
\begin{equation}
\binom{2k}{k}e(\xi)+ p_k(\xi)=\frac{2}{(2\pi)^{2k}(k!)^2}(|(R^k)^+_+|^2-|(R^k)^-_+|^2)\epsilon
\end{equation}

\begin{equation}
\binom{2k}{k}e(\xi)- p_k(\xi)=\frac{2}{(2\pi)^{2k}(k!)^2}(|(R^k)^-_-|^2-|(R^k)^+_-|^2)\epsilon.
\end{equation}
Thus if $(R^k)^-_+=0$, then 
\begin{equation}\label{GThorpe1}
\binom{2k}{k}e(\xi)+ p_k(\xi)\geq 0.
\end{equation}
While if $(R^k)^+_-=0$, then 
\begin{equation}\label{GThorpe2}
\binom{2k}{k}e(\xi)- p_k(\xi)\geq 0.
\end{equation}
Therefore if $(R^k)^-_+=0$ and $(R^k)^+_-=0$ then 
\begin{equation}\label{Thorpe}
\binom{2k}{k}e(\xi)\geq |p_k(\xi)|.
\end{equation}
When $\xi=TM$, $h=g$, and $D=\nabla_g$ is the Levi-Civita connection, $(R^k)^-_+=0$ if and only if $(R^k)^+_-=0$, so we only see the inequality~(\ref{Thorpe}) in this case.

\end{Th}

\begin{Rmk} Note the inequalities make sense since $\Lambda^{4k}T^*M$ is an oriented real line bundle.

The inequality~(\ref{Thorpe}) when $\xi=TM$, $h=g$, and $D=\nabla_g$ is the original Thorpe inequality in \cite{Thorpe}. An interesting aspect of the above theorem is that the original Thorpe inequality is revealed to be the conjunction of the two independent inequalities~(\ref{GThorpe1}), (\ref{GThorpe2}).
\end{Rmk}

\begin{proof}
We only need to explain why $(R^k)^-_+=0$ if and only if $(R^k)^+_-=0$ when $\xi=TM$, $h=g$, and $D=\nabla_g$. The underlying reason for this is the pair-interchange symmetry $R(W, X, Y, Z)=R(Y, Z, W, X)$ of a Riemann curvature tensor $R$, which immediately implies that $R^k$ has the interchange symmetry $$R^k(X_1, ..., X_{2k}, Y_1, ..., Y_{2k})=R^k(Y_1, ..., Y_{2k}, X_1, ..., X_{2k}).$$ That is, $R^k\in C^{2k}V=S^2(\Lambda^{2k}T^*M)$. Therefore it suffices to show that if $\dim(V)=4k$ and $T\in C^{2k} V$ then $T^+_-=0$ if and only if $T^-_+=0$.

Let $*:\Lambda^{2k}V\rightarrow \Lambda^{2k}V$, $*_1=*\otimes 1: \Lambda^{2k}V\otimes \Lambda^{2k}V\rightarrow \Lambda^{2k}V\otimes \Lambda^{2k}V$, and $*_2=1\otimes *: \Lambda^{2k}V\otimes \Lambda^{2k}V\rightarrow \Lambda^{2k}V\otimes \Lambda^{2k}V$. Then $T^+_-=\tfrac{1}{4}(1+*_1)(1-*_2)T$ and $T^-_+=\tfrac{1}{4}(1-*_1)(1+*_2)T$. Let $\{e_1, ..., e_{4k}\}$ be an oriented orthonormal basis, and denote any $e_{i_1}\wedge...\wedge e_{i_{2k}}$ by $e_I$. We also employ the notation $e_{I'}=\star e_I$; $I'$ is thus a representative of the set of indices complementary to $I$ which are ordered in such a way that when $I$ and $I'$ are concatenated they constitute an even permutation of $[4k]$. Then $T=T^{IJ}e_I\otimes e_J$ where the sum is over the $e_I$ such that $i_1<...<i_{4k}$, and $T^{IJ}=T^{JI}$ since $T\in C^{2k}V$. Then using $(I')'=I$, which comes from $*^2=1$,
\begin{align*}
T^+_-&=\tfrac{1}{4}T^{IJ}(e_I+e_{I'})\otimes (e_J-e_{J'})\\
&=\tfrac{1}{4}T^{IJ}e_I\otimes e_J-\tfrac{1}{4}T^{IJ}e_I\otimes e_{J'}+\tfrac{1}{4}T^{IJ}e_{I'}\otimes e_J-\tfrac{1}{4}T^{IJ}e_{I'}\otimes e_{J'}\\
&=\tfrac{1}{4}T^{IJ}e_I\otimes e_J-\tfrac{1}{4}T^{IJ'}e_I\otimes e_{J}+\tfrac{1}{4}T^{I'J}e_{I}\otimes e_J-\tfrac{1}{4}T^{I'J'}e_{I}\otimes e_{J}\\
&=\tfrac{1}{4}(T^{IJ}-T^{IJ'}+T^{I'J}-T^{I'J'})e_I\otimes e_J
\end{align*}
while similarly
\begin{align*}
T^-_+&=\tfrac{1}{4}T^{IJ}(e_I-e_{I'})\otimes (e_J+e_{J'})\\
&=\tfrac{1}{4}T^{IJ}e_I\otimes e_J+\tfrac{1}{4}T^{IJ}e_I\otimes e_{J'}-\tfrac{1}{4}T^{IJ}e_{I'}\otimes e_J-\tfrac{1}{4}T^{IJ}e_{I'}\otimes e_{J'}\\
&=\tfrac{1}{4}T^{IJ}e_I\otimes e_J+\tfrac{1}{4}T^{IJ'}e_I\otimes e_{J}-\tfrac{1}{4}T^{I'J}e_{I}\otimes e_J-\tfrac{1}{4}T^{I'J'}e_{I}\otimes e_{J}\\
&=\tfrac{1}{4}(T^{IJ}+T^{IJ'}-T^{I'J}-T^{I'J'})e_I\otimes e_J.
\end{align*}
Now just notice that $$(T^+_-)^{IJ}=\tfrac{1}{4}(T^{IJ}-T^{IJ'}+T^{I'J}-T^{I'J'})=\tfrac{1}{4}(T^{JI}-T^{J'I}+T^{JI'}-T^{J'I'})=(T^-_+)^{JI}$$ which immediately implies the conclusion.

\end{proof}

\section{Further Commentary}

Thorpe's original presentation of his inequality goes further by giving a pleasing geometric interpretation of the vanishing conditions $(R^k)^-_+=(R^k)^+_-=0$. This comes about in the following way. Let $T\in D^{2k}V=\Lambda^{2k}V\otimes \Lambda^{2k}V$ where $\dim{V}=4k$, $*=*_{2k}\otimes *_{2k}$, and write $T=T^+_+ + T^+_- + T^-_+ + T^-_-$. Then $*T=T^+_+ - T^+_- - T^-_+ + T^-_-$, so $*T=T$ if and only if $T^+_-=T^-_+=0$. When $k=1$ (so the manifold is four-dimensional), Thorpe and Singer in \cite{SingerThorpe} had already found that $R=*R$ if and only if $K(P)=K(P^{\perp})$ where $K(P)$ denotes the sectional curvature $K$ of a two-plane $P$ (and these conditions are equivalent to the four-manifold being Einstein). The key observations are that $R$ and $*R$ are the ``same kind of tensor'' in that they both have the symmetries of a Riemann curvature tensor, and that as a consequence of the Bianchi identity such tensors are completely determined by their sectional curvatures. Therefore, the natural argument to extend this geometric interpretation to general values of $k$ would be to establish that $R^k$, $*R^k$ both satisfy certain symmetry conditions, including an analogue of the Bianchi identity, and that these symmetries are sufficiently strong that these tensors are determined by their sectional curvatures. We will make some comments about the details of this argument as they appear in \cite{Thorpe}.

\subsection{The Bianchi Identity for Double Forms}

Thorpe proves in \cite{Thorpe} that the complete antisymmetrization of $R^p$ vanishes, i.e. 
\begin{equation}\label{alt}
R^p_{[i_1...i_{4p}]}=0
\end{equation}
and refers to this as the Bianchi identity for $R^p$. He then remarks that due to the alternating properties of $R^p$ in the first $2p$ and last $2p$ variables, (\ref{alt}) may be rewritten as
\begin{equation}\label{Bianchi}
\sum_{l=1}^{2p+1} (-1)^l R^p_{i_1...\hat{i_l} ... i_{2p+1} i_l j_1 ... j_{2p-1}}=0.
\end{equation}
It is this second identity that has come to be called \emph{the Bianchi identity} for $R^p$ by later authors, e.g. \cite{Kulkarni}, \cite{Labbi}. 

Now, both equations (\ref{alt}), (\ref{Bianchi}) are true for $R^p$. (The identity (\ref{alt}) follows from (\ref{Bianchi}); see immediately after Lemma (\ref{BianchiAlt}) below. For a simple proof of (\ref{Bianchi}) see \cite{Kulkarni} or \cite{Labbi}.) The claim that they are equivalent is, however, false in general, as we will demonstrate. (Ironically, Thorpe had intended this equivalence to correct a previous faulty proof that $R^p$ satisfies (\ref{Bianchi}). But it does seem that he was eventually aware of this secondary error; see below.)

First of all, one should invoke not only the alternating symmetries in the first and last $2p$ variables, but also the symmetry that allows for these variables to be interchanged, i.e. one should use that $R^p\in C^{2p}T^*M=S^2(\Lambda^{2p}T^*M)$. Presumably, Thorpe intended for this additional symmetry to be used, for one invokes each of these symmetries to show that (\ref{alt}) and (\ref{Bianchi}) \emph{are} equivalent when $p=1$ (this is straightforward to check); what this shows is that the Bianchi identity for a Riemann curvature tensor is equivalent to the vanishing of the full antisymmetrization of that tensor. Perhaps this equivalence when $p=1$ is what misled Thorpe into thinking it was true in general.

Thus we take Thorpe's claim to be that $T\in C^{2p}V^*$ where $\dim{V}\geq 2p$ has vanishing complete antisymmetrization if and only if $$\sum_{l=1}^{2p+1} (-1)^l T_{i_1...\hat{i_l} ... i_{2p+1} i_l j_1 ... j_{2p-1}}=0.$$ The reverse implication is true, but the forward implication is the claim that is false in general. I made a note about this in passing in my dissertation; later I found that A. Stehney in \cite{Stehney} notes (just before Lemma 1.4 of that work) that she observed the same in her dissertation \cite{StehneyDiss}. However, I was unable to obtain a copy of that work. It is worth noting that \cite{StehneyDiss} was written under Thorpe's supervision, so it seems he was ultimately aware of his mistake in \cite{Thorpe}.

Now, to see that the reverse implication of Thorpe's claim is true, it suffices to notice $\sum_{l=1}^{p+1} (-1)^l T_{i_1...\hat{i_l} ... i_{p+1} i_l j_1 ... j_{q-1}}=0$ is equivalent to $T_{[i_1... i_{p+1}]j_1...j_{q-1}}=0$ by the following

\begin{Lemma}\label{BianchiAlt} If $T\in D^{p, q}V^*=\Lambda^{p}V^*\otimes \Lambda^{q}V^*$   then 
\begin{equation}\label{identity}
T_{[i_1... i_{p+1}]j_1...j_{q-1}}=\frac{(-1)^{p+1}}{p+1} \sum_{l=1}^{p+1} (-1)^l T_{i_1...\hat{i_l} ... i_{p+1} i_l j_1 ... j_{q-1}}.
\end{equation}
\end{Lemma}
The desired reverse implication follows from this since $T_{[i_1... i_{p+1}]j_1...j_{q-1}}=0$ obviously implies $T_{[i_1... i_{p+1}j_1...j_{q-1}]}=0$. (To wit, the vanishing of a tensor after antisymmetrization over a set of arguments implies vanishing of the tensor after antisymmetrization over any superset of the same arguments.)

To prove (\ref{identity}), for any term on the left-hand side consider the associated permutation $\sigma\in S_{p+1}$ and denote $\sigma(p+1)=l$. There are $q!$ many such permutations, and dividing by the $(p+1)!$ in the definition of the antisymmetrization accounts for the inverse factor of $p+1$. Then merely use the antisymmetry of $T$ in the first $p$ arguments to arrange them in the order $i_1...\hat{i_l} ... i_{p+1}$; the additional factor of $(-1)^{p+1}$ gives the proper sign of $(-1)^{p+1-l}$ to the permutation associated to $i_1...\hat{i_l} ... i_{p+1}i_l$.

We will now show that the forward implication of Thorpe's claim is false in general, that is we have
\begin{Prop}\label{ThorpeF}
If $\dim{V}\geq 8$, there exists $T\in C^4V^*$ such that $T_{[i_1... i_{5}j_1j_2j_{3}]}=0$ but $T_{[i_1... i_{5}]j_1j_2j_{3}}\neq 0$.
\end{Prop}

\noindent The basic underlying reason for this proposition follows from a simple count. There are $(4p)!$ terms in the expression for $T_{[i_1... i_{2p+1}j_1...j_{2p-1}]}$ and \emph{in general} this number will be cut down by a factor of $2[(2p)!]^2$ on account of the symmetries of $T$, leaving $\frac{1}{2}\binom{4p}{2p}$ independent terms. On the other hand, there are only $2p+1$ independent terms in general in the expression for $T_{[i_1... i_{2p+1}]j_1...j_{2p-1}}$ after accounting for the antisymmetry in the first $2p$ arguments, and $\frac{1}{2}\binom{4p}{2p}>2p+1$ for large $p$. 

\begin{proof}
Consider a basis $\{e_1,... , e_8, ...\}$ of $V$. Begin to define $T$ by declaring that $T_{i_1...i_8}=0$ if $i_1...i_8$ is not a permutation of $1...8$. We then have to define $T_{i_1...i_4j_1j_2j_38}$ where $i_1...i_4j_1j_2j_3$ is a permutation of $1...7$, $i_1<...<i_4$, and $j_1<j_2<j_3$. If we do so then the other components of $T$, that is the $T_{i_1...i_8}$ where $i_1...i_8$ is a permutation of $1...8$, are uniquely determined by requiring that $T\in S^2(\Lambda^{4}V^*\otimes \Lambda^{4}V^*)$, and we will define them by this unique determination. To this end, among all such $T_{i_1...i_4j_1j_2j_38}$, we declare $T_{12345678}=T_{12374568}=1$, and set the rest equal to $0$.

To see that $T$ so defined satisfies $T_{[i_1... i_8]}=0$, note first that this is trivially satisfied by the definition of $T$ if the indices aren't a permutation of $1...8$. Thus we just need to verify that $T_{[1...8]}=0$. Observe that any term in the expression $T_{[1...8]}$ may be transformed using the symmetries of $T$ into exactly one of the $T_{i_1...i_4j_1j_2j_38}$ terms where $i_1...i_4j_1j_2j_3$ is a permutation of $1...7$, $i_1<...<i_4$, and $j_1<j_2<j_3$ and there are exactly $2(4!)^2$ such terms corresponding to a fixed $T_{i_1...i_4j_1j_2j_38}$ since the indices $i_1...i_4j_1j_2j_38$ are distinct. It follows that $T_{[1...8]}$ is proportional to $\sum (-1)^{\mathrm{sgn}(i_1...i_4j_1j_2j_38)}T_{i_1...i_4j_1j_2j_38}=T_{12345678}-T_{12374568}=1-1=0$.

On the other hand $T_{[1...5]678}$ is proportional to the single term $T_{1...8}=1$ and therefore doesn't vanish, which completes the proof.
\end{proof}

We end this subsection with the following proposition, which gives a correct statement about an equivalent form of (\ref{alt}).

\begin{Prop}
If $T\in C^{2p}V^*$ then
\begin{equation}\label{altalt}
T_{[i_1...i_{4p}]}=T_{[i_1...i_{4p-1}]i_{4p}}
\end{equation}
\end{Prop}

\begin{proof}
Each term in the sum defining $T_{[i_1...i_{4p}]}$ is associated to a permutation $\sigma\in S_{4p}$. We can divide these permutations into $4p$ groupings, each determined by which $l\in [4p]$ is sent by $\sigma$ to $4p$. For any such grouping, use the symmetries of $T$ to fix a permutation that maneuvers $i_{4p}$ into the last index slot: that is, use the transpose symmetry of the first and last group of $2p$ indices and the alternating symmetry within these groups. The transpose symmetry of $T$ doesn't involve incurring a minus sign, which matches the positive sign of a permutation which interchanges two groups of even size; this is where we use the fact that $2p$ is even in our assumption $T\in C^{2p}V^*$. Each time the alternating symmetries of $T$ are used, we of course incur a minus sign, which matches the negative sign of the corresponding permutation. Thus all signs are accounted for properly. Once $i_{4p}$ is put in the last slot what fills out the other slots is a completely arbitrary permutation of $i_1$, ..., $i_{4p-1}$ corresponding to an element of $\tau\in S_{4p-1}$. This occurs for each of the $4p$ groupings, and the resulting factor of $4p$ cancels the one in the inverse factor of $(4p)!$, leaving an inverse factor of $(4p-1)!$ as in the definition of $T_{[i_1...i_{4p-1}]i_{4p}}$, which gives the proposition. 
\end{proof}

\subsection{The Bianchi-Star Identity}

As discussed at the beginning of this section, identifying (\ref{Bianchi}) as the correct analogue for $R^k$ of the classical Bianchi identity is an important step in giving a geometric characterization of $(R^k)^+_-=(R^k)^-_+=0$, or equivalently $R^k=*R^k$. The key remaining step is that curvature structures which satisfy the Bianchi identity are determined by their sectional curvatures; see Prop. 2.1 of \cite{Kulkarni} for a precise statement. Utilizing this proposition, however, requires that $*R^k$ satisfy the Bianchi identity. Thorpe argues in \cite{Thorpe} that the full antisymmetrization of $*R^k$ vanishes, but as we have seen, this is not equivalent to the Bianchi identity (\ref{Bianchi}). I am not aware of a source that fills this gap, so we will do so here.

The following definition can be found in \cite{Kulkarni}, \cite{Labbi}:

\begin{Def}
The \emph{Bianchi operator} $b: D^{p,q}V^*\rightarrow D^{p+1, q-1}V^*$ is defined by
\begin{equation}
b(T)_{i_1...i_{p+1}j_1...j_{q-1}}=(-1)^{p+1}(p+1)T_{[i_1...i_{p+1}]j_1...j_{q-1}}=\sum_{l=1}^{p+1} (-1)^l T_{i_1...\hat{i_l} ... i_{p+1} i_l j_1 ... j_{q-1}}
\end{equation}
where we have made use of (\ref{BianchiAlt}) in the second equality.

Alternatively, on basis elements, where $e^{i_1...i_k}=e^{i_1}\wedge ... \wedge e^{i_k}$, we have
\begin{equation}\label{BianchiBasis}
b(e^{i_1...i_p}\otimes e^{j_1...j_q})=\sum_{l=1}^q (-1)^l e^{j_li_1...i_p}\otimes e^{j_1...\hat{j_l}...j_q}
\end{equation}
\end{Def}

For the alternative definition on basis elements, see \cite{Kulkarni}; the formula is easily verified by a computation.

We will now prove an identity that shows that the transpose followed by the star operating on double-forms preserves the kernel of the Bianchi operator; an immediate corollary is that $*R^k$ satisfies the Bianchi identity since $R^k$ is a symmetric double-form which satisfies the Bianchi identity.

\begin{Prop}
If $\dim{V}=n$ is an inner product space, $b$ is the Bianchi operator, $*$ is the star operator on double-forms, and $t:D^{p,q}V^*\rightarrow D^{q,p}V^*$ denotes the transpose of tensor factors, then
\begin{equation}\label{BianchiStar}
t*b=(-1)^{p+q-1}b*t
\end{equation}
as maps $D^{p,q}V^*\rightarrow D^{n-q+1, n-p-1}V^*$.
\end{Prop}

\begin{proof}
It suffices to check that this holds on basis elements $\omega=e^{i_1...i_p}\otimes e^{j_1...j_q}=e^I\otimes e^J$. We assume that the $e^i$ are orthonormal and that $\{e^1,..., e^n\}$ defines the orientation and thus the star operator of $V$; however, as discussed above, the arbitrary choice of orientation does not affect the resulting $*$ operator on double forms. Let $K=I\cap J$ and write $I=K\cup L$, $J=K\cup M$, where $K$, $L$, $M$ are disjoint. We can now assume $\omega=e^{KL}\otimes e^{KM}$. Denote also $H=[n]-(K\cup L\cup M)$. Defining $|K|=a$, $|L|=b$, $|M|=c$, $|H|=d$, we have $a+b=p$, $a+c=q$, and $a+b+c+d=n$.

We first compute
$$(*t)(\omega)=(*t)(e^{KL}\otimes e^{KM})=*(e^{KM}\otimes e^{KL})=(-1)^{KMHL}(-1)^{KLHM}e^{HL}\otimes e^{HM}$$
where $(-1)^{ABCD}$ denotes the sign of the permutation of $[n]$ which is represented by reading the indices in $A$, $B$, $C$, $D$ in order. Thus
\begin{align*}
(b*t)(\omega)&=(-1)^{KMHL}(-1)^{KLHM}\sum_{i=1}^c (-1)^{d+i}e^{m_ih_1...h_dl_1...l_b}\otimes e^{h_1...h_dm_1...\hat{m_i}...m_c}\\
&=(-1)^{KMHL}(-1)^{KLHM}\sum_{i=1}^c (-1)^{d+i}e^{\{m_i\}HL}\otimes e^{H(M-\{m_i\})}
\end{align*}
by (\ref{BianchiBasis}). 

On the other hand, to compute $(t*b)(\omega)$, we first calculate
\begin{align*}
b\omega&=\sum_{i=1}^c (-1)^{a+i} e^{m_ik_1...k_al_1...l_b}\otimes e^{k_1...k_am_1...\hat{m_i}...m_c}\\
&=\sum_{i=1}^c (-1)^{a+i} e^{\{m_i\}KL}\otimes e^{K(M-\{m_i\})}
\end{align*}
Now
\begin{align*}
*e^{\{m_i\}KL}&=(-1)^{\{m_i\}KLH(M-\{m_i\})}e^{H(M-\{m_i\})}\\
&=(-1)^{a+b+d+i-1}(-1)^{KLHM}e^{H(M-\{m_i\})}
\end{align*}
while
\begin{align*}
*e^{K(M-\{m_i\})}&=(-1)^{c-i}(-1)^{KMHL}e^{\{m_i\}HL}
\end{align*}
so
\begin{align*}
(t*b)(\omega)&=(-1)^{KMHL}(-1)^{KLHM}\sum_{i=1}^c (-1)^{a+i+a+b+d+i-1+c-i}e^{\{m_i\}HL}\otimes e^{H(M-\{m_i\})}\\
&=(-1)^{KMHL}(-1)^{KLHM}\sum_{i=1}^c (-1)^{(d+i)+(a+b)+(a+c)-1} e^{\{m_i\}HL}\otimes e^{H(M-\{m_i\})}\\
&=(-1)^{p+q-1}(b*t)(\omega)
\end{align*}
\end{proof}

\begin{Cor}
If $R$ denotes a Riemann curvature tensor then $*R^k$ satisfies the Bianchi identity (\ref{Bianchi}).
\end{Cor}

\begin{proof}
Plug $R^k$ into (\ref{BianchiStar}) (with $p=q=2k$) and use the fact that $t(R^k)=R^k$.
\end{proof}

\end{document}